\documentclass{rstransa}

\jname{rsta}
\Journal{Phil. Trans. R. Soc}

\usepackage{amsmath}
\usepackage{amssymb}
\usepackage{amsfonts} 
\usepackage{amsthm}
\usepackage{booktabs}
\usepackage{cleveref}

\usepackage{thmtools}
\usepackage{stmaryrd} 
\usepackage{comment} 

\usepackage{pifont}
\newcommand{\cmark}{\ding{51}}%
\newcommand{\xmark}{\ding{55}}%

\usepackage{enumitem}
\setenumerate{label=(\roman*),itemsep=3pt}

\declaretheorem[
name = Theorem,
]{theorem}
\declaretheorem[
name = Corollary,
sibling = theorem
]{corollary}
\declaretheorem[
name = Lemma,
sibling = theorem
]{lemma}
\declaretheorem[
name = Proposition,
sibling = theorem
]{proposition}

\declaretheorem[
name = Claim,
numbered = no
]{claim*}

\usepackage{etoolbox}
\AtEndEnvironment{proof}{\setcounter{claim}{0}}

\declaretheoremstyle[%
  spaceabove=-6pt,%
  spacebelow=6pt,%
  headfont=\normalfont\itshape,%
  postheadspace=1em,%
  qed=$\blacksquare$,%
  headpunct={.}
]{mystyle}

\declaretheorem[
name = Remark,
sibling = theorem,
style=definition
]{remark}

\declaretheorem[
name = Definition,
sibling = theorem,
style=definition
]{definition}

\declaretheorem[
name = Question,
sibling = theorem,
style=definition
]{question}

\newcommand{\HA}{\mathsf{HA}}
\newcommand{\IKP}{\mathsf{IKP}}
\newcommand{\CZF}{\mathsf{CZF}}

\newcommand{\CZFER}{\mathsf{CZF_{ER}}}
\newcommand{\IZF}{\mathsf{IZF}}
\newcommand{\IZFR}{\mathsf{IZF_R}}
\newcommand{\ZFC}{\mathsf{ZFC}}

\newcommand{\ECST}{\mathsf{ECST}}
\newcommand{\BCST}{\mathsf{BCST}}
\newcommand{\AC}{\mathsf{AC}}
\newcommand{\MP}{\mathsf{MP}}
\newcommand{\PowerSet}{\mathsf{PowerSet}}

\newcommand{\CPC}{\mathbf{CPC}}
\newcommand{\CQC}{\mathbf{CQC}}
\newcommand{\IPC}{\mathbf{IPC}}
\newcommand{\IQC}{\mathbf{IQC}}

\newcommand{\Prop}{\mathsf{Prop}}

\newcommand{\sent}{\mathsf{sent}}

\newcommand{\Logic}{\mathbf{L}}
\newcommand{\QLogic}{\mathbf{QL}}

\DeclareMathOperator{\dom}{dom}

\newcommand{\seq}[1]{\langle #1 \rangle}
\newcommand{\Seq}[2]{\langle #1 \, | \, #2 \rangle}
\newcommand{\set}[1]{\{ #1 \}}
\newcommand{\Set}[2]{\{ #1 \, | \, #2 \}}

\renewcommand{\phi}{\varphi}
\renewcommand{\epsilon}{\varepsilon}
\newcommand{\E}{\exists}

\usepackage{scalerel,stackengine}
\stackMath
\newcommand\reallywidehat[1]{%
\savestack{\tmpbox}{\stretchto{%
  \scaleto{%
    \scalerel*[\widthof{\ensuremath{#1}}]{\kern-.6pt\bigwedge\kern-.6pt}%
    {\rule[-\textheight/2]{1ex}{\textheight}}
  }{\textheight}%
}{0.5ex}}%
\stackon[1pt]{#1}{\tmpbox}%
}

\usepackage{ifthen}

\newcommand{\vsim}{\mathrel{\scalebox{1}[1.75]{$\shortmid$}\mkern-3.1mu\raisebox{0.15ex}{$\sim$}}}
\newcommand{\vsimpr}{\vsim^{\raisebox{-.5ex}{\rm\scriptsize 1}}}

\usepackage{turnstile}
\usepackage{adjustbox}

\renewcommand{\makehor}[4]
  {\ifthenelse{\equal{#1}{n}}{\hspace{#3}}{}
   \ifthenelse{\equal{#1}{s}}{\rule[-0.5#2]{#3}{#2}}{}
   \ifthenelse{\equal{#1}{d}}{\setlength{\lengthvar}{#2}
     \addtolength{\lengthvar}{0.5#4}
     \rule[-\lengthvar]{#3}{#2}
     \hspace{-#3}
     \rule[0.5#4]{#3}{#2}}{}
   \ifthenelse{\equal{#1}{t}}{\setlength{\lengthvar}{1.5#2}
     \addtolength{\lengthvar}{#4}
     \rule[-\lengthvar]{#3}{#2}
     \hspace{-#3}
     \rule[-0.5#2]{#3}{#2}
     \hspace{-#3}
     \setlength{\lengthvar}{0.5#2}
     \addtolength{\lengthvar}{#4}
     \rule[\lengthvar]{#3}{#2}}{}
   \ifthenelse{\equal{#1}{w}}{
     \setbox0=\hbox{$\sim$}%
     \raisebox{-.6ex}{\hspace*{-.05ex}\adjustbox{width=#3,height=\height}{\clipbox{0.75 0 0 0}{\usebox0}}}}{}
  }

\usepackage{blindtext}

\begin{document}

\title{Logics and Admissible Rules of Constructive Set Theories}
\author{Rosalie Iemhoff$^1$ and Robert Passmann$^2$}
\address{$^1$ Department of Philosophy and Religious Studies, Utrecht University, The Netherlands \\ $^2$ Institute for Logic, Language and Computation, University of Amsterdam, The Netherlands}

\subject{Logic}
\corres{Corresponding author here}

\begin{abstract}
    We survey the logical structure of constructive set theories and point towards directions for future research. Moreover, we analyse the consequences of being extensible for the logical structure of a given constructive set theory. We finally provide examples of a number of set theories that are extensible. 
\end{abstract}

\begin{fmtext}
\section{Introduction}

Constructive set theories are formal systems in which we can conduct intuitionistic or constructive mathematics. Many systems of constructive set theory have strong connections to type theory and are interesting because they allow us to analyse the computational content of mathematical statements.

When a constructive mathematical theory $T$ is defined, we usually pay special attention to ensure that not all instances of the law of excluded middle, $$p \vee \neg p,$$ are derivable in $T$. Even if an axiom looks constructive, it may happen that it entails this logical law, and if all instances of the law of excluded middle are derivable, then the axiom is constructively unacceptable. 

Consider, for a well-known example, the famous theorem by Diaconescu \cite{Diaconescu1975} and Goodman-Myhill \cite{GoodmanMyhill1978} that adding the axiom of choice to intuitionistic set theory $\IZF$ results in classical $\ZFC$ set theory. This result is often taken as evidence for rejecting the axiom of choice as constructively valid. However, showing that a certain instance of excluded middle is not constructively valid is not sufficient for ensuring that only principles of intuitionistic logic are satisfied. After all, there could be intermediate principles that are derivable in $T$. 
Dick de Jongh was the first to investigate this phenomenon: in his doctoral dissertation \cite{DeJongh1970}, he showed that the propositional principles valid in Heyting Arithmetic $\HA$ under all substitutions are exactly those of propositional \phantom{intuitionistic logic $\IPC$.}
\end{fmtext}

\maketitle

\noindent intuitionistic logic $\IPC$. This fact is now known as \textit{de Jongh's Theorem} and has sparked fruitful investigations into the logical structure of intuitionistic arithmetic. A generalisation of de Jongh's Theorem would be to consider not only the propositional principles of $\HA$ but also its \textit{propositional admissible rules} (see \Cref{Definition: admissibility relation}). Visser \cite{Visser1999} showed that the propositional admissible rules of $\HA$ are exactly those of $\IPC.$ A further generalisation would be to consider the two notions in the setting of predicate logic. 

In this article, we survey results about the logical structure of constructive set theories and also provide a few new results in the area as well as questions for future research.

The study of the logical properties of classical and nonclassical theories goes back at least several decades. Early results on the logics of theories date from the 1970s, and Rybakov's seminal work on admissibility \cite{Rybakov} marked the beginning of a period during which the admissible rules were investigated for many theories, see \cite{Iemhoff2015} for an overview. What is characteristic of that period is that the theories involved were mainly arithmetical theories, such as Heyting Arithmetic, or extensions of propositional logics, such as modal, intermediate and substructural logics. This naturally leads to the more general questions studied today, questions that ask whether the phenomena observed for arithmetic apply to other well-known constructive theories, and whether the phenomena observed for propositional logics also hold for predicate logics. Behind all this lie conceptual issues as well: What does it mean for a theory to have the same logic or rules as a given theory? Should it be required of a genuine constructive theory that it has the same admissible rules as intuitionistic propositional or predicate logics? These questions characterize the modern perspective on the area. 
We do not know the answers to all of them yet, but the results in this paper provide some answers for constructive set theories.

\paragraph{Overview} We will introduce what we understand as the logical structure of a theory in \Cref{Section: The logical structure of a theory}, and then survey the techniques and results obtained so far in \Cref{Section: techniques and results}. \Cref{Section: Logics and rules of set theories} will be dedicated to obtaining a few new results. We mention suggestions for future research throughout the paper but also list some more in \Cref{Section: questions}. Finally, we discuss the logics and admissible rules of the extreme case of classical theories in \Cref{Section: classical case}.

\section{Preliminaries: the logical structure of a theory}
\label{Section: The logical structure of a theory}

We denote intuitionistic propositional logic by $\IPC$ and intuitionistic first-order logic by $\IQC$. Classical propositional logic is denoted by $\CPC$ and classical first-order logic by $\CQC$. If a logic $J$ satisfies $\IPC \subseteq J \subseteq \CPC$ or $\IQC \subseteq J \subseteq \CQC$, then it is a (propositional or first-order) \textit{intermediate} logic.

In this paper, we work with first-order logics formulated in the language consisting of countably many predicate symbols $P_i$ for $i \in \mathbb{N}$ in every arity. The formulas of first-order logic are recursively defined in the usual way with the connectives and quantifiers $\neg$, $\wedge$, $\vee$, $\rightarrow$, $\forall$ and $\exists$. We also consider $\bot$ and $\top$ to be part of the logical vocabulary. Note that we take constants and functions to be expressed as predicates in the usual way. The language of propositional logic is obtained by restricting to nullary predicates (i.e. the propositional letters) and forgetting the quantifiers. Finally, we assume that all first-order languages share the same set of variables $\set{x_i}_{i \in \mathbb{N}}$. 

\begin{definition}
    Let $T$ be a theory. A map $\sigma$ is a \textit{$T$-assignment} if and only if it maps each predicate symbol $R$ to a $T$-formula $\sigma(R)$ such that the number of free variables of $\sigma(R)$ is not bigger than the arity of $R$.
\end{definition}

For example, given a binary relation symbol $R$, we could obtain a $T$-assignment by setting $\sigma(R) = \text{``} x_0 \in x_1 \text{''}$. We denote the formulas of first-order logic by $\mathcal{L}_{\mathsf{FOL}}^\mathsf{form}$ and those of a theory $T$ by $\mathcal{L}_{T}^\mathsf{form}$. We
can then obtain substitutions from assignments as follows. 

\begin{definition}
    Let $T$ be a theory. 
    A map $\sigma: \mathcal{L}_{\mathsf{FOL}}^\mathsf{form} \to \mathcal{L}_T^\mathsf{form}$ is a \textit{$T$-substitution} if it obtained from a $T$-assignment $\tau$ such that $\sigma(R(x_{i_1},\dots,x_{i_n})) = \tau(R)(x_{i_1},\dots,x_{i_n})$, and, moreover, commutes with $\bot$, $\top$, and the logical  connectives. If the range of $\sigma$ is $\set{\top,\bot}$, then it is a \textit{ground substitution}.
\end{definition}

The theory $T$ is often clear from the context and we will then speak of \textit{substitutions} instead of \textit{$T$-substitutions}. As usual, a rule $\Gamma / \Delta$ is a pair of finite sets of logical formulas. 

\begin{definition}
    \label{Definition: admissibility relation}
    Let $T$ be a theory. A rule $\Gamma / \Delta$ is \textit{admissible in $T$, $\Gamma \vsimpr_T \Delta$,} if and only if for every substitution $\sigma$, if $T \vdash \sigma(A)$ for all $A \in \Gamma$, then there is $B \in \Delta$ such that $T \vdash \sigma(B)$.
\end{definition}

By $\vsim$ we will denote the admissibility relation in which we consider only propositional formulas. We will consider more closely now the propositional and first-order tautologies of a theory.

\begin{definition}
    A formula $A$ of first-order logic is a \textit{tautology of $T$} if and only if $\top \vsimpr_T A$.
\end{definition}

\begin{proposition}
    \label{Proposition: tautology characterisation}
    A formula $A$ is a tautology of $T$ if and only if $T \vdash \sigma(A)$ for all substitutions $\sigma$.
\end{proposition}
\begin{proof}
    Note that $\top \vsimpr_T A$ is by definition equivalent to $T \vdash \sigma(\top)$ implying $T \vdash \sigma(A)$ for all $\sigma$. As the premise is true for any $T$, the implication is equivalent to $T \vdash \sigma(A)$ for every $\sigma$.
\end{proof}

Given this definition, we can define the propositional and first-order logics of a theory.

\begin{definition}
    Let $T$ be a theory. The propositional logic of $T$, $\Logic(T)$, consists of all propositional tautologies of $T$. The first-order logic of $T$, $\QLogic(T)$, consists of all first-order tautologies of $T$.
\end{definition}

We can now ask four natural questions about the logical structure of a given theory $T$:
\begin{enumerate}
    \item What are the first-order admissible rules of $T$? ${\vsimpr_T} = {\textbf{?}}$
    \item What are the propositional admissible rules of $T$? ${\vsim_T} = {\textbf{?}}$
    \item What is the first-order logic of $T$? 
    ${\QLogic(T)} = {\textbf{?}}$
    \item What is the propositional logic of $T$? ${\Logic(T)} = {\textbf{?}}$
\end{enumerate}

Given the definitions above, it is clear that answers to some of these questions imply answers to others: If we know an answer to (i), then we know answers to (ii)-(iv). If we know an answer to (ii), then to (iv) as well. Finally, if we know an answer to (iii), then also to (iv). Still, there are several reasons to consider these questions separately. First, the following theorem shows that an informative answer---in the sense of a recursive description---to (i) is in general not possible.

\begin{theorem}[Visser {\cite{Visser1999}}]
    \label{Theorem: predicate rules pi2}
    Let $T$ be a theory such that $I\Delta_0 + \mathrm{Exp}$ can be relatively interpreted in $T$. If $T$ has the disjunction property, then $\vsimpr_T$ is $\Pi^0_2$-complete.
\end{theorem}
\begin{proof}
    This is Visser's result; however, as we are not working with relative interpretations here, one needs to add this explicitly on the logical side by replacing the quantifiers $\exists x A(x)$ and $\forall x A(x)$ with $\exists x (R(X) \wedge A(x))$ and $\forall x (R(x) \rightarrow A(x))$, respectively, where $R$ is a new relation symbol.
\end{proof}

However, one may of course still wonder whether the first-order admissible rules of any two theories are the same (see \Cref{Question: T T' vsimpr}) or study the admissibility of specific interesting rules (see, for example, van den Berg and Moerdijk \cite{vandenBergMoerdijk2012}). Another reason to consider these questions separately is that, of course, the more complex objects require more complicated proofs than the simpler ones. In a sense, question (iv) can often be seen as a \textit{litmus test} for the more difficult questions (ii) and (iii): if we conjecture that $\QLogic(T) = \IQC$, then showing that $\Logic(T) = \IPC$ is a step in the right direction (though there are counter examples, e.g., $\Logic(\IZF) = \IPC$ but $\QLogic(\IZF) \neq \IQC$, see \Cref{Section: techniques and results}). It may even be necessary to know an answer to (iv) before applying certain techniques for obtaining an answer to (ii), viz. the propositional admissible rules of $T$. We will see much more detail in the remainder of this paper. For now, we will dive a bit deeper into the logics and propositional admissible rules of a given theory.

We say that a theory $T$ is \textit{based on a logic $J$} if $T$ is axiomatised over $J$. As the results discussed in this paper show, if a theory $T$ is based on, for example, $\IQC$, we do not necessarily have that $\QLogic(T) = \IQC$. Crucially, we cannot only study the logics of a theory $T$ based on intuitionistic logic but, given an intermediate logic $J$, we can consider the theory $T(J)$ obtained as the closure of $T$ under $J$. De Jongh, Verbrugge and Visser \cite{deJonghVerbruggeVisser2011} studied these theories and defined the \textit{de Jongh property} as follows.

\begin{definition}
    Let $J$ be a propositional or first-order intermediate logic. A theory $T$ has the \textit{de Jongh property for $J$} just in case $\Logic(T(J)) = J$ if $J$ is propositional or $\QLogic(T(J)) = J$ if $J$ is first-order.  
\end{definition}

We also say that a theory $T$ satisfies \textit{de Jongh's theorem} if $\Logic(T) = \IPC$, and that $T$ satisfies \textit{de Jongh's theorem for first order logic} if $\QLogic(T) = \IQC$. To illustrate research in this area, we now give an incomplete history of de Jongh's theorem for arithmetical theories (see de Jongh, Verbrugge and Visser \cite{deJonghVerbruggeVisser2011} for a more complete history). De Jongh \cite{DeJongh1970} proved that $\Logic(\HA) = \IPC$---in other words, he proved that $\HA$ satisfies the de Jongh property for $\IPC$. Leivant \cite{Leivant1979} showed using proof-theoretic means that $\QLogic(\HA) = \IQC$; van Oosten later gave a model-theoretic proof of the same fact. De Jongh, Verbrugge and Visser \cite{deJonghVerbruggeVisser2011} proved that $\HA$ has the de Jongh property for several classes of intermediate logics. The following propositions are immediate from \Cref{Proposition: tautology characterisation}. 

\begin{proposition}
    Let $T \subseteq S$ be theories in the same language. If $A$ is a tautology of $T$, then it is a tautology of $S$. In particular, $\Logic(T) \subseteq \Logic(S)$ and $\QLogic(T) \subseteq \QLogic(S)$. \qed
\end{proposition}

\begin{proposition}
    Let $J$ be a propositional or first-order intermediate logic. If $T \subseteq S$ are theories in the same language and $S$ satisfies the \textit{de Jongh property} for $J$, then so does $T$. In particular, if $S$ satisfies de Jongh's theorem, then so does $T$. \qed
\end{proposition}

A first crucial observation on the admissible rules of a theory is that these are bounded by its logic.

\begin{theorem}[Visser {\cite{Visser1999}}]
    \label{Theorem: Bound on admissible rules}
    Let $T$ be a theory. If $A \vsimpr_T B$, then $A \vsimpr_{\QLogic(T)} B$; i.e. ${\vsimpr_T} \subseteq {\vsimpr_{\QLogic(T)}}$. Similarly, if $A \vsim_T B$, then $A \vsim_{\Logic(T)} B$; i.e. ${\vsim_T} \subseteq {\vsim_{\Logic(T)}}$.
\end{theorem}

Not much is known about the converse direction. A counterexample could be obtained with \Cref{Theorem: predicate rules pi2} if it turns out that the predicate admissible rules of $\IQC$ are of complexity lower than $\Pi^0_2$-completeness. We close these preliminaries with two particularly helpful results for studying the propositional admissible rules of a given theory. A theory $T$ is called \textit{extensible} if any Kripke model of $T$ can be extended by adding a new root and corresponding domain to obtain a new model of $T$ (see also \Cref{Definition: extensibility}).

\begin{theorem}[Visser {\cite[Lemma 4.1]{Visser1999}}]
    \label{Theorem: Extensible admissibility}
    Let $T$ be a theory with $\Logic(T) = \IPC$. If $T$ is extensible, then the propositional admissible rules of $T$ are exactly those of $\IPC$, ${\vsim_T} = {\vsim_\IPC}$.
\end{theorem}

The {\it Visser rule} $V_n$ is the following rule for propositional formulas $A_i$, $B_i$ and $C$:
\begin{equation*}
    \frac{(\bigwedge_{i = 1}^n (A_i \rightarrow B_i) \rightarrow (A_{n+1} \vee A_{n+2})) \vee C}{\bigvee_{j=1}^{n+2}(\bigwedge_{i=1}^n (A_i \rightarrow B_i) \rightarrow A_j) \vee C}
    \tag{$V_n$}
\end{equation*}
The collection of {\it Visser rules} consists of the rules $V_n$ for every $n$. 

\begin{theorem}
    \label{Theorem: Vissers rules admissibility}
    Let $T$ be a theory with $\Logic(T) = \IPC$. If the Visser rules are admissible in $T$, then the admissible rules of $T$ are exactly those of $\IPC$. 
\end{theorem}
\begin{proof}
    Let $A$ and $B$ be propositional formulas. By \Cref{Theorem: Bound on admissible rules}, it suffices to show that if $A \vsim_{\IPC} B$, then $A \vsim_T B$. So assume that $A \vsim_{\IPC} B$. By a theorem of Iemhoff \cite[Theorem 3.9]{Iemhoff2005}, this means that $A \vdash_\IPC^\mathsf{V} B$, where $\vdash_\IPC^\mathsf{V}$ denotes the derivability relation of $\IPC$ extended with all of Visser's rules. In other words, there is a proof tree, potentially using $A$ as a premise, all steps in which are instances of the rules of $\IPC$ and $\mathsf{V}$, and whose conclusion is $B$. Given that Visser's rules are admissible in $T$, and using the fact that $T$ is based on intuitionistic logic, it is straightforward to see that every rule application in the proof tree is admissible in $T$. Hence, $B$ is an admissible consequence of $A$, $A \vsim_T B$.
\end{proof}

The two theorems give two routes to proving that the propositional admissible rules of a theory $T$ are exactly those of $\IPC$. We will see below that \Cref{Theorem: Extensible admissibility} can be seen as an instance of \Cref{Theorem: Vissers rules admissibility} because Visser's rules are admissible in any extensible theory (\Cref{Lemma: T extensible implies Visser admissible}). In this section, we covered the preliminaries for theories based on intuitionistic logic. There is more to say for the extreme case of theories based on classical logic, see \Cref{Section: classical case}.

\section{Survey on techniques and results}
\label{Section: techniques and results}

\subsection{Constructive set theories}
\label{Subsection: Constructive set theories}

\begin{figure}
    \centering
    {\scriptsize
    \begin{align*}
        &\forall x \forall y (x = y \leftrightarrow \forall z (z \in x \leftrightarrow z \in y)) \tag{Extensionality} \\
        &\forall x \forall y \exists z \forall w (w \in z \leftrightarrow w = y \vee w = z) \tag{Pair} \\
        &\forall x \exists y \forall z (z \in y \leftrightarrow \exists w \in x (z \in w)) \tag{Union} \\
        &\exists x \forall y (y \notin x) \tag{Empty Set} \\
        &\exists x (\emptyset \in x \wedge \forall y(y \in x \rightarrow y \cup \set{y} \in x)) \tag{Infinity} \\
        &\forall x \exists y \forall z (z \in y \leftrightarrow (z \in x \wedge \phi(z))) \tag{Separation} \\
        &\forall x (\forall y \in x \exists z \phi(y,z) \rightarrow \exists w \forall y \in x \exists z \in w \phi(y,z)) \tag{Collection} \\
        &\forall x (\forall y \in x \exists ! z \phi(y,z) \rightarrow \exists w \forall y \in x \exists z \in w \phi(y,z)) \tag{Replacement} \\
        &\forall x (\forall y \in x \exists z \phi(y,z) \rightarrow \exists w (\forall y \in x \exists z \in w \phi(y,z) \wedge \forall z \in w \exists y \in x \phi(y,z)) \tag{Strong Collection} \\
        &\forall x \exists y \forall w (w \in y \leftrightarrow w \subseteq x) \tag{Power Set} \\
        &\forall x \forall y \exists z \forall w (\forall a \in x \exists b \in y \phi(a,b,w) \rightarrow \exists c \in z (\forall a \in x \exists b \in c \phi(a,b,w) \wedge \forall b \in c \exists a \in x \phi(a,b,w))) \tag{Subset Collection} \\
        &\forall x \forall y \exists z (\text{``$z$ is full in $x$ and $y$''}) \tag{Fullness} \\
        &(\forall x (\forall y \in x \phi(y)) \rightarrow \phi(x) ) \rightarrow (\forall x \phi(x)) \tag{$\in$-induction} \\
        & \forall x \forall y \exists z \forall w (w \in z \leftrightarrow \text{``$w$ is a function $x \to y$''}) \tag{Exponentiation} \\
        &\exists x [(\emptyset \in x \wedge \forall y(y \in x \rightarrow y \cup \set{y} \in x)) \wedge (\forall z (\emptyset \in z \wedge \forall y(y \in z \rightarrow y \cup \set{y} \in z)) \rightarrow x \subseteq z)] \tag{Strong Infinity} \\
        & \forall x [(\forall y (y \in x \rightarrow \exists z (z \in y))) \rightarrow \exists f: x \to \bigcup x \ f(x) \in x]
        \tag{Axiom of Choice}
    \end{align*}
    }
    \caption{Axioms of set theory.}
    \label{fig:axioms}
\end{figure}

In this section, we can only give a very brief introduction to constructive set theory. As usual, the language $\mathcal{L}_\in$ of set theory is a first-order language with equality and a single binary relation symbol `$\in$' to denote set membership. We consider the bounded quantifiers $\forall x \in a \ \phi(x)$ and $\exists x \in a \ \phi(x)$ to be abbreviations for $\forall x (x \in a \rightarrow \phi(x))$ and $\exists x (x \in a \wedge \phi(x))$, respectively. The crucial aim of intuitionistic and constructive set theories is to provide a set-theoretic foundation for mathematics on the basis of intuitionistic instead of classical logic. We will now introduce the relevant systems.

In \cref{fig:axioms}, we have spelt out all the axioms of set theory, giving rise to the following theories. Note that ``$z$ is full in $x$ and $y$'' means that every element of $z$ is a total relation between $x$ and $y$ and for every total relation $w$ between $x$ and $y$, there is some $u \in z$ such that $u \subseteq w$. We can now define the relevant theories. The \textit{bounded separation} consists of all instances of (Separation) where $\phi$ is a bounded formula (i.e. $\phi$ is $\Delta_0$). In a similar way, we obtain the axiom scheme of \textit{bounded collection}.

\begin{definition}
    We define the following set theories on the basis of intuitionistic first-order logic.
    \begin{enumerate}
        \item \textit{Intuitionistic Zermelo-Fraenkel Set Theory}, $\IZF$, consists of the axioms and schemes of extensionality, empty set, union, pairing, infinity, separation, $\in$-induction, collection and power set.
        \item \textit{Constructive Zermelo-Fraenkel Set Theory}, $\CZF$ consists of the axioms and schemes of extensionality, empty set, union, pairing, strong infinity, bounded separation, $\in$-induction, strong collection and subset collection.
        \item \textit{Intuitionistic Kripke-Platek set theory}, $\IKP$, consists of the axioms and schemes of extensionality, empty set, union, pairing, infinity, bounded separation, $\in$-induction, and bounded collection.
        \item \textit{Basic constructive set theory}, $\BCST$, consists of the axioms and schemes of extensionality, empty set, union, pairing, replacement and bounded separation.
        \item     \textit{Elementary constructive set theory}, $\ECST$, consists of the axioms and schemes of extensionality, empty set, pairing, union, strong infinity, bounded separation and replacement.
    \end{enumerate}
    The theory $\IZFR$ is $\IZF$ with replacement instead of collection, and $\CZFER$ is $\CZF$ with exponentiation and replacement instead of strong collection and subset collection.
\end{definition}

Note that $\CZF$ is usually formulated with the axiom of infinity instead of strong infinity. Obviously, the axiom of strong infinity implies infinity. The converse is also true on the basis of $\CZF$ without any infinity axioms {\cite[Section 2.1]{Aczel2011LCSTAID}}; so we can consider $\CZF$ to be formulated with the axiom of strong infinity instead of infinity, as this will be important later. For a detailed development of the mathematics of and in constructive set theories, we refer the reader to the notes of Aczel and Rathjen \cite{AczelRathjen2010}. Whenever we discuss results for an arbitrary set theory $T$ in this paper, we assume that $T$ contains at least intuitionistic logic. 

\subsection{Techniques and results}

We will now survey what is known about the logical structure of constructive set theories. For many years, only the following two negative results were known.

\begin{theorem}[Diaconescu {\cite{Diaconescu1975}}; Goodman and Myhill {\cite{GoodmanMyhill1978}}]
    \label{Theorem: Diaconescu-Goodman-Myhill}
    Let $T$ be a set theory that proves $0 \neq 1$ as well as the axioms of extensionality, pairing, and the separation scheme. If $T$ proves the axiom of choice, then $T$ proves all instances of the law of excluded middle.
\end{theorem}

\begin{theorem}[Friedman and Ščedrov {\cite{FriedmanScedrov1986}}]
    \label{Theorem: Friedman-Scedrov}
    Let $T$ be a set theory containing the axioms and schemes of extensionality, separation, pairing and (finite) union. Then $\IQC \subsetneq \QLogic(T)$. In particular, $\IQC \subsetneq \QLogic(\IZF) \subsetneq \CQC$.
\end{theorem}

While we do not wish to spell out the proofs of these theorems in detail, we will point out here that both of them make crucial use of the Separation scheme to define sets that allow to derive a logical scheme for all formulas. This observation is crucial as we will see in a while that Passmann \cite{Passmann2022} showed that both theorems fail on the basis of $\CZF$, i.e. when the set theory does not contain the full separation scheme. Note that \Cref{Theorem: Friedman-Scedrov} leads to the following open question.
\begin{question}
    Let $T$ be a set theory, based on intuitionistic logic, satisfying the conditions of \Cref{Theorem: Friedman-Scedrov}. What is $\QLogic(T)$? In particular, what is $\QLogic(\IZF)$?
\end{question}

Before considering more recent results, we will briefly recall some notation and results on Kripke models. A \textit{Kripke frame $(K,\leq)$} consists of a partial order $\leq$ on a set $K$. A \textit{valuation $V$ on $(K,\leq)$} is a function $V$ assigning sets of propositional letters to nodes such that $v \leq w$ and $p \in V(v)$ entails $p \in V(w)$. A \textit{Kripke model $(K,\leq,V)$} consists of a Kripke frame $(K,\leq)$ and a valuation $V$ on $(K,\leq)$. The valuation $V$ can be extended to the forcing relation $\Vdash$ between Kripke models and valuation Kripke models and propositional formulas in the usual way. The defining property of valuations is called \textit{persistence} and transfers to all propositional formulas in Kripke models, i.e. $v \Vdash \phi$ and $v \leq w$ entails $w \Vdash \phi$ (this property also holds for all first-order formulas in the case of first-order models, see below). We refer to the literature for standard results about Kripke models for intuitionistic logic.

We are now ready to move to more recent results. All the results we are going to mention are obtained by model-theoretic methods, i.e. a result of the form $\QLogic(T) = J$ is usually obtained by, first, showing that $T$ proves all instances of $J$ (this is usually easy), and second, for every logical principle $A$ such that $J \not \vdash A$, one constructs a countermodel $M$ of $T$ which also fails that principle. In other words, there is then a $T$-substitution $\sigma$ such that $M \vDash T$ but $M \not \vDash \sigma(A)$. We say that an intermediate logic $J$ is \textit{characterised} by a class $K$ of Kripke models if and only if $J \vdash A$ if and only if $M \vDash A$ for every $M \in K$. The class of \textit{finite trees} consists of the finite partial orders $(P,\leq)$ such that for any $p \in P$ the set $\Set{q \in P}{q \leq p}$ is linearly ordered.

\begin{theorem}[Passmann {\cite{Passmann2020}}]
    Let $T \subseteq \IZF$ be a set theory. If $J$ be an intermediate propositional logic that is characterised by a class of finite trees, then $\Logic(T(J)) = J$. In particular, $\Logic(T) = \IPC$.
\end{theorem}

This result was proved by using so-called \textit{blended} Kripke models, and we will see an adaptation of this technique below in \Cref{Section: Logics and rules of set theories}. For now, we just note that Passmann's blended models are inspired by the models that Lubarksy \cite{Lubarsky2005} used to show various independence results around $\CZF$.

The next two results were proved by Iemhoff and Passmann \cite{IemhoffPassmann2021}. Note that $\IKP^+$ is obtained by adding weak versions of strong collection and subset collection to $\IKP$ (for details see their paper). An intermediate logic $J$ is called \textit{Kripke-complete} if there is a class $K$ of Kripke models such that $J$ is characterised by $K$.

\begin{theorem}[Iemhoff and Passmann {\cite{IemhoffPassmann2021}}]
    Let $T \subseteq \IKP^+ + \MP + \AC$ be a set theory. If $J$ is a Kripke-complete intermediate propositional logic, then $\Logic(T(J)) = J$.
\end{theorem}

\begin{theorem}[Iemhoff and Passmann {\cite{IemhoffPassmann2021}}]
    Let $T \subseteq \IKP^+ + \MP + \AC$ be a set theory. If $J$ is a Kripke-complete intermediate first-order logic contained in the least transitive model of $\ZFC + V = L$, then $\QLogic(T(J)) = J$. In particular, $\QLogic(T) = \IQC$.
\end{theorem}

For these results, Iemhoff and Passmann use (what they call) Kripke models \textit{with classical domains}. These models are obtained by equipping each node of a Kripke frame with a classical model of set theory in a coherent way. An earlier result by Iemhoff \cite{Iemhoff2010} entails that these models will always satisfy $\IKP$. Passmann showed in his master's thesis \cite{Passmann2018} that this result is in a certain sense optimal: there are Kripke models with classical domains that do not satisfy the exponentiation axiom (which is a consequence of $\CZF$). The above results could be proved for large classes of logics because the simple structure of the Kripke models with classical domains allows much control about their logical structure.

Finally, Iemhoff and Passmann also obtained the following negative result for first-order logic with equality. Note that $\QLogic^=(T)$ is obtained just like $\QLogic(T)$ with the additional requirement that subsitutions commute with `$=$'.
\begin{theorem}[{\cite[Theorem 61]{IemhoffPassmann2021}}]
    Let $T$ be a set theory containing the axioms of extensionality, empty set and pairing. Then the first-order logic with equality of $T$, $\QLogic^=(T)$ , is strictly stronger than $\IQC^=$, i.e., $\IQC^= \subsetneq \QLogic^=(T)$.
\end{theorem}

\begin{question}
    Given any set theory $T$, what is its first-order logic with equality $\QLogic^=(T)$?
\end{question}

The following results were obtained by combining realisability techniques with transfinite computability. The latter is a generalisation of classical computability by allowing machines to run for an infinite amount of time and/or use an infinite amount of space. For a thorough introduction, we refer the reader to Carl's book \cite{Carl2019}.

\begin{theorem}[Carl, Galeotti and Passmann {\cite{CarlGaleottiPassmann2020}}]
    The propositional admissible rules of $\IKP$ are exactly those of $\IPC$, i.e. ${\vsim_\IKP} = {\vsim_\IPC}$.
\end{theorem}

\begin{theorem}[Passmann {\cite{Passmann2022}}]
    The propositional admissible rules of $\CZF$ are exactly those of $\IPC$, i.e. ${\vsim_\CZF} = {\vsim_\IPC}$.
\end{theorem}

\begin{theorem}[Passmann {\cite{Passmann2022}}]
    \label{Theorem: Passmann 2022 Main Result}
    Let $T \subseteq \CZF + \PowerSet + \AC$ be a set theory. The first-order logic of $T$ is intuitionistic first-order logic $\IQC$, $\QLogic(T) = \IQC$.
\end{theorem}

It is crucial to note here that the proof method of \Cref{Theorem: Passmann 2022 Main Result} does not allow to prove the result  for any intermediate logic $J$. 

\begin{question}
    Can \Cref{Theorem: Passmann 2022 Main Result} be generalised to a class of intermediate first-order logics?
\end{question}

In general, all results surveyed in this section introduce assumptions on the logic $J$ under consideration. We may therefore ask the following question.

\begin{question}
    Is it possible to extend the results of this section to larger classes of logics? If not, find counterexamples for which these theorems fail if the assumptions on the logics $J$ are weakened.
\end{question}

\begin{table}
    {
    \footnotesize
    \centering
    \begin{tabular}{lllll}
        \toprule
        $T$ & 
        $\Logic(T) = \IPC$ & 
        $\QLogic(T) = \IQC$ & 
        ${\vsim_T} = {\vsim_\IPC}$ 
        \\
        \midrule
        $\IKP$ & 
        \cmark & 
        \cmark & \cmark  
        \\
        $\BCST$ & 
        \cmark & 
        \cmark & 
        \cmark (\Cref{Theorem: IZFR CZFER ECST BCST extensible})
        \\
        $\ECST$ & 
        \cmark & 
        \cmark & 
        \cmark (\Cref{Theorem: IZFR CZFER ECST BCST extensible})
        \\
        $\CZFER$ & 
        \cmark & 
        \cmark &
        \cmark (\Cref{Theorem: IZFR CZFER ECST BCST extensible})
        \\
        $\CZF$ & 
        \cmark & 
        \cmark & \cmark  
        \\
        $\CZF + \AC$ & \cmark 
        & \cmark 
        & \cmark 
        \\
        $\IZFR$ & 
        \cmark & 
        \xmark; $\IQC \subsetneq \QLogic(\IZFR) = {\textbf{?}} \subsetneq \CQC$ 
        & \cmark (\Cref{Theorem: IZFR CZFER ECST BCST extensible})
        \\
        $\IZF$ & 
        \cmark & 
        \xmark; $\IQC \subsetneq \QLogic(\IZF) = {\textbf{?}} \subsetneq \CQC$ &
        \xmark; ${\vsim_\IPC} \subsetneq {\vsim_\IZF} = {\textbf{?}}$
        \\
        $\IZF + \AC$ &
        \xmark; $\Logic(\IZF + \AC) = \CPC$ & 
        \xmark; $\CQC \subseteq \QLogic(\IZF + \AC)$ & 
        \xmark; ${\vsim_{\IZF + \AC}} = {\vsim_\CPC}$ 
        \\
        \bottomrule
    \end{tabular}
    }
    \caption{An overview of the most important intuitionistic and constructive set theories and whether they satisfy de Jongh's propositional and first-order theorems, and whether their admissible rules are exactly those of $\IPC$, as discussed in the survey. The question marks (``$\textbf{?}$'') indicate open problems. The results marked with a reference within this paper are new.}
    \label{Table: Overview}
\end{table}

\section{Logics and rules of set theories}
\label{Section: Logics and rules of set theories}

\subsection{Logic, rules, and the extension property}

In this section, we will prove a set-theoretic analogue of Visser's \cite[Lemma 4.1]{Visser1999} \textit{Main Lemma}. Given a Kripke frame $K$, we write $K^+$ for the frame extended with a new root. The construction of adding a new root to a Kripke model was first used by Smorynski \cite{Smorynski1973} for models of $\HA$ to give an alternative proof of de Jongh's theorem for $\HA$. In the arithmetical case, it suffices to equip the new root with the standard model of arithmetic. The case of set theory requires a more elaborate construction, as we will now see.

\begin{definition}
    \label{Definition: extensibility}
    Let $\Gamma$ be a set of sentences in the language of set theory. A set of sentences $\Delta$ is \emph{$\Gamma$-extensible} if for every Kripke model $M \Vdash \Gamma \cup \Delta$ of set theory with underlying frame $K$, there is a model $M^+$ based on $K^+$ such that $M^+ \upharpoonright K = M$ and $M^+ \Vdash \Delta$; $\Delta$ is extensible if it is $\emptyset$-extensible. 
\end{definition}

To avoid cumbersome notation, we will say that a sentence $\phi$ is ($\Gamma$-)extensible just in case $\set{\phi}$ is ($\Gamma$-)extensible. We will later need the following two brief observations.

\begin{lemma}
    \label{Lemma: Extensible 1}
    Let $\Gamma \subseteq \Delta$ be sets of formulas. If $\Delta$ is $\Gamma$-extensible, then $\Delta$ is extensible.
\end{lemma}
\begin{proof}
    Let $M \Vdash \Delta$. As $\Gamma \subseteq \Delta$, we have $M \Vdash \Delta \cup \Gamma$. By $\Gamma$-extensibility, it follows that $M^+ \Vdash \Delta$. 
\end{proof}
    
\begin{lemma}
    \label{Lemma: Extensible 2}
    Let $\Gamma \subseteq \Delta$ be sets of formulas. If a sentence $\phi$ is $\Gamma$-extensible, then it is $\Delta$-extensible.
\end{lemma}
\begin{proof}
    Let $M \Vdash \set{\phi} \cup \Delta$. As $\Gamma \subseteq \Delta$, we have $M \Vdash \set{\phi} \cup \Gamma$. By $\Gamma$-extensibility, $M^+ \Vdash \phi$.
\end{proof}

A Kripke frame $(K,\leq)$ is a \textit{finite splitting tree} if $K$ is finite, connected, and every $v \in K$ has either no successors or at least two immediate successors. A node $v \in K$ with no successors is also called a \textit{leaf}. An easy induction on the height of finite splitting trees allows to show that every node in such a tree is uniquely determined by the set of leafs above it. A set theory $T$ is called \textit{subclassical} if there is a classical model of $T$.

The following theorem follows the same idea as Smorynski's proof of de Jongh's Theorem for $\HA$ \cite{Smorynski1973}. For this reason, this proof method is also sometimes referred to as \textit{Smorynski's trick}. 

\begin{theorem}
    \label{Theorem: Extensibility De Jongh Property}
    Let $J$ be a propositional intermediate logic characterised by a class of finite splitting trees. If $T$ is a subclassical recursively enumerable extensible set theory, then $J$ is the propositional logic of $T(J)$, i.e. $\Logic(T(J)) = J$.
\end{theorem}
\begin{proof}
    Let $2^{<\omega}$ be the set of binary sequences of finite length, and $2^n$ be the set of binary sequences of length $n$. If $T$ is a recursively enumerable theory, let $\Gamma_T$ be its Gödel sentence. Let $\phi^{\seq{0}} := \Gamma_T$, and $\phi^{\seq{1}} := \neg \Gamma_T$. Clearly both $T + \phi^{\seq{0}}$ and $T + \phi^{\seq{1}}$ are consistent by Gödel's Incompleteness Theorem. By recursion on the length of $s \in 2^{<\omega}$, we define 
    \begin{align*}
        \phi^{s^\smallfrown{\seq{0}}} &:= \phi^s \wedge \Gamma_{T + \phi^s}, \text{ and,} \\
        \phi^{s^\smallfrown{\seq{1}}} &:= \phi^s \wedge \neg \Gamma_{T + \phi^s}.
    \end{align*}
    By inductively applying Gödel's Incompleteness Theorem, it follows that $T + \phi^s$ is consistent for every $s \in 2^{<\omega}$; so, for every $s \in 2^{<\omega}$, let $M_s$ be a classical model such that $M_s \vDash T + \phi^s$.
    
    We now observe that, given $s, t \in 2^{<\omega}$ of the same length with $s \neq t$, it must be that $\phi^s$ and $\phi^t$ are jointly inconsistent: Let $i$ be minimal such that $s(i) \neq t(i)$. Then $s \upharpoonright i = t \upharpoonright i$ and we can assume, without loss of generality, that $s(i) = 0$ and $t(i) = 1$. The sentences $\phi^s$ and $\phi^t$ are defined as conjunctions in such a way that $\phi^s$ contains the conjunct $\Gamma_{T + \phi^{s\upharpoonright i}}$, and $\phi^t$ contains the conjunct $\neg\Gamma_{T + \phi^{t\upharpoonright i}}$. Since $s \upharpoonright i = t \upharpoonright i$, it follows that $\phi^s \rightarrow \neg \phi^t$. We can conclude for $n < \omega$ and $s \in 2^n$ that
    $$
        M_s \vDash \phi^s \wedge \bigwedge_{t \in 2^n \setminus \set{s}} \neg \phi^t.
    $$
    
    Let $T$ be a set theory and $J$ a logic, as given in the statement of the theorem. To prove that $\Logic(T(J)) = J$, we will proceed as follows: Let $\mathcal{C}$ be a class of finite splitting trees that characterises the logic $J$. It is clear that $J \subseteq \Logic(T(J))$. To show that $\Logic(T(J)) \subseteq J$, we proceed by contraposition. So assume that $J \not \vdash A$ for some propositional formula $A$, then there is a finite splitting tree $(K,\leq) \in \mathcal{C}$ and a valuation $V$ on $K$ such that $(K,\leq,V) \Vdash J$ but $(K,\leq,V) \not \Vdash A$. On the basis of this propositional Kripke model, we will construct a Kripke model $M \Vdash T$ of set theory and a propositional translation $\tau$ such that $M \not \Vdash A^\tau$. As $M$ is based on the frame $(K,\leq) \Vdash J$, it follows that $M \Vdash T(J)$. Hence, it follows that $A \notin \Logic(T(J))$.
    
    As every finite splitting tree can be constructed from its set of leaves by iterating the operation of adding a new root, we can obtain a model of $T$ on the frame $(K,\leq)$ as follows. Find $n < \omega$ such that there are at least as many $s \in 2^n$ as there are leaves in $K$; let $l \mapsto s_l$ be an injective map assigning sequences to leaves. Assign the models $M_{s_l}$ to the leaves of $(K,\leq)$ and use the extensibility of $T$ to construct a Kripke model $M$ of $T$ with underlying frame $(K,\leq)$. 
    Given $v \in K$, let $E_v$ be the set of leaves $l \geq v$. Then consider the formula:
    $$
        \gamma_v := \neg \neg \bigvee_{l \in E_v} \phi^{s_l}.
    $$
    Now recall that every node $v$ in the finite splitting tree $(K,\leq)$ is characterised by the set $E_v$, and that $M_{s_l} \vDash \phi^{s_k}$ if and only if $l = k$. It follows that $M, w \Vdash \gamma_v$ if and only if $w \geq v$. 
    
    We are now ready to construct the translation $\tau$. Given a propositional letter $p$, let
    \begin{align*}
    \tau(p) :&= \bigvee_{v \in V(p)} \gamma_v = \bigvee_{v \in V(p)} \neg \neg \bigvee_{l \in E_v} \phi^{s_l}.
    \end{align*}
    Moreover, $\tau$ commutes with the propositional connectives $\neg$, $\rightarrow$, $\vee$, and $\wedge$. 
    
    To finish the proof of the theorem, it now suffices to show that $M, w \Vdash \tau(p)$ if and only if $w \in V(p)$; it then follows by induction that $M, w \Vdash A^\tau$ if and only if $(K,\leq,V), w \Vdash A$. So assume that $M, w \Vdash \tau(p)$. Equivalently, $M, w \Vdash \gamma_v$ for some $v \in V(p)$. We have seen that this is equivalent to $w \geq v$ for some $v \in V(p)$, which, by persistence, holds if and only if $w \in V(p)$.
\end{proof}

Note that $\IPC$ is characterised by the class of finite trees \cite[Theorem 6.12]{TroelstraVanDalen1988VI}. To show that $\IPC$ is characterised by the class of finite splitting tress, note that duplicating branches of a tree does not change the formulas satisfied at the root.

\begin{corollary}
    \label{Corollary: Extensibility implies De Jongh}
    Let $T$ be a subclassical recursively enumerable set theory. If $T$ is extensible, then the propositional logic of $T$ is intuitionistic propositional logic $\IPC$, i.e. $\Logic(T) = \IPC$.
\end{corollary}

Let us first observe the following helpful fact. Recall that a theory $T$ has the disjunction property whenever $T \vdash \phi \vee \psi$ implies $T \vdash \phi$ or $T \vdash \psi$.

\begin{lemma}
    \label{Lemma: DP}
    If $T$ is an extensible set theory, then $T$ has the disjunction property.
\end{lemma}
\begin{proof}
    By contraposition. Assume that $T \not \vdash \phi$ and $T \not \vdash \psi$, then there are models $M_0$ and $M_1$ of $T$ such that $M_0 \not \Vdash \phi$ and $M_1 \not \Vdash \psi$. Let $M$ be the disjoint union of these models, then $M^+ \Vdash T$ as $T$ is extensible. Moreover, persistence implies that $M^+ \not \Vdash \phi \vee \psi$, hence $T \not \vdash \phi \vee \psi$.
\end{proof}

\begin{lemma}
    \label{Lemma: T extensible implies Visser admissible}
    If $T$ is an extensible set theory, then Visser's rules are admissible in $T$.
\end{lemma}
\begin{proof}
    By \Cref{Lemma: DP}, it is sufficient to show that the following rules $V'_n$ are admissible:
    \begin{equation*}
        \frac{\bigwedge_{i = 1}^n (A_i \rightarrow B_i) \rightarrow (A_{n+1} \vee A_{n+2})}{\bigvee_{j=1}^{n+2}(\bigwedge_{i=1}^n (A_i \rightarrow B_i) \rightarrow A_j)}
        \tag{$V'_n$}
    \end{equation*}
    The admissibility of these rules is a standard argument and proceeds as follows by contraposition. For greater readability we write $\sigma A$ for $\sigma(A)$ in this proof. 
    
    Let $\sigma: \Prop \to \mathcal{L}_\in^\sent$ be any substitution and assume that $T \not \vdash \bigvee_{j=1}^{n+2}(\bigwedge_{i=1}^n (\sigma A_i \rightarrow \sigma B_i) \rightarrow \sigma A_j)$, i.e., $T \not \vdash \bigwedge_{i=1}^n (\sigma A_i  \rightarrow \sigma B_i ) \rightarrow \sigma A_j $ for $j = 1, \dots, n+2$. By Kripke-completeness, let $M_j$ be a Kripke model of $T$ with root $r_j$ such that $M_j \not \Vdash \bigwedge_{i=1}^n (\sigma A_i  \rightarrow \sigma B_i ) \rightarrow \sigma A_j $. In this situation, we can assume without loss of generality that $M_j, r_j \Vdash \bigwedge_{i=1}^n (\sigma A_i  \rightarrow \sigma B_i )$, but $M_j, r_j \not \Vdash \sigma A_j $. 
    
    Now, let $M$ be the disjoint union $\Seq{M_j}{j = 1, \dots, n+2}$ of the models. As $T$ is extensible, consider a model $M^+$ extending $M$ with a new root $r$ such that $M^+ \Vdash T$. By persistence, $r \not \Vdash \sigma A_j $ for all $j = 1, \dots, n+2$. Hence, $r \Vdash \sigma A_i  \rightarrow \sigma B_i $ for all $i = 1, \dots, n$, but $r \not \Vdash \sigma A_{n+1}  \vee \sigma A_{n+2} $. Hence, $T \not \vdash \bigwedge_{i = 1}^n (\sigma A_i  \rightarrow \sigma B_i ) \rightarrow (A_{n+1}  \vee \sigma A_{n+2} )$.
    
    We may conclude that the rule $V'_n$ is admissible in $T$.
\end{proof}

The following theorem is a direct consequence of the previous \Cref{Corollary: Extensibility implies De Jongh}, \Cref{Lemma: T extensible implies Visser admissible} and \Cref{Theorem: Vissers rules admissibility}.

\begin{theorem}
    \label{Theorem: Extensibility implies admissible rules}
    Let $T$ be a subclassical recursively enumerable set theory. If $T$ is extensible, then the admissible rules of $T$ are exactly those of $\IPC$, i.e. ${\vsim_T} = {\vsim_\IPC}$. \qed
\end{theorem}

Finally, the following corollary is immediate from \Cref{Theorem: Extensibility De Jongh Property}.

\begin{corollary}
    Let $T$ be a subclassical recursively enumerable extensible set theory, and $\phi$ be a sentence in the language of set theory. If $\phi$ is $T$-extensible, then $\Logic(T + \phi) = \IPC$.
\end{corollary}

In conclusion, the axiom of choice, $\AC$, is not $\IZFR$-extensible.

\begin{question}
    Is the axiom of choice $\CZF$-extensible? 
\end{question}

\subsection{Extensible set theories}

In this section, we assume the existence of a proper class of inaccessible cardinals (see \Cref{Remark: Alternatives for inaccessibles}). We will prove the extension property for a variety of constructive set theories by providing one particular construction for extending a given model. This construction is an adaptation of Passmann's \cite{Passmann2020} blended models. A \textit{Kripke model for set theory} is a first-order Kripke model of the form $(K, \leq, \set{D_v}_{v \in K},\set{f_{vw}}_{v \leq w \in K},\set{E_v}_{v \in K})$, where $\set{D_v}_{v \in K}$ is a collections of domains, $\set{f_{vw}}_{v \leq w \in K}$ a collection of transition functions between domains, and $\set{E_v}_{v \in K}$ a collection of interpretations of the `$\in$'-predicate. Note that the $D_v$ are assumed to be sets. Of course, the transition functions and interpretations satisfy the usual coherence conditions (for more details see, for example, \cite{IemhoffPassmann2021}).

\begin{definition}
    Let $M = (K, \leq, \set{D_v}_{v \in K},\set{f_{vw}}_{v \leq w \in K},\set{E_v}_{v \in K})$ be a Kripke model of set theory. The \textit{extended model} $M^+ = (K^+,\leq,\set{D_v}_{v \in K^+},\set{f_{vw}}_{v \leq w \in K^+},\set{E_v}_{v \in K^+})$ is defined as follows:
    \begin{enumerate}
        \item Let $r \notin K$, the so-called \textit{new root}. Then extend $K$ and $\leq$ as follows: 
        $$K^+ = K \cup \set{r} \text{ and } r \leq v \text{ for all } v \in K^+$$
        \item The domains $D_v$ for $v \in K$ are already given. The domain $D_r$ at the new root is defined inductively as follows:
      \begin{enumerate}
         \item $D_r^0 = \emptyset$,
         \item $D_r^\alpha$ consists of functions $x$ with $\dom(x) = K^+$ such that:
          \begin{enumerate}
                \item $x(r) \subseteq \bigcup_{\beta < \alpha} D_r^\beta$,
                \item $x(v) \in D_v$ for all $v \in K$,
                \item if $y \in x(r)$, then $y(v) E_v x(v)$ for all $v \in K$,
                \item if $v, w \in K$ with $v \leq w $, then $x(w) = f_{vw}(x(v))$.
            \end{enumerate}
            \item $D_r = \bigcup_{\alpha \in \kappa} D_r^\alpha$, where $\kappa$ is the least inaccessible cardinal $\kappa > |K| + \sum_{v \in K}|D_v|$.
        \end{enumerate}
        \item The membership relation $E_v$ at $v \in K$ are already defined; at the new root $r$, define by $y E_r x$ if and only if $y \in x(r)$
        \item The transition functions $f_{vw}:D_v \to D_w$ are already defined for $v, w \in K$ with $v \leq w$. Given $v \in K$, define $f_{rv}: D_r \to D_v$ by $f_{rv}(x) = x(v)$.
    \end{enumerate}
\end{definition}

To help the reader digest this construction, we will give a simple example and provide some general intuition for the construction. Let $M$ be any classical model for set theory. In other words, $M$ is a one point Kripke model; call its single point $v$. The extended model $M^+$ has then a new root $r$. An element of the root $D_r$---a set at the root---is a function $x$ with domain $\set{r,v}$ such that $x(r) \in D_r$ and $x(v) \in D_v$. Moreover, if $y \in x(r)$, then $y(v) \in x(v)$. Intuitively, a set $x$ at node $r$ may thus already contain some elements at the root $r$ but may also collect new elements when transitioning to $v$. Another way to think of this construction idea is as adding a new root whose elements are approximations of the sets already existing in the model we are starting from.

The next step is to observe that many set-theoretical axioms are extensible. For convenience we abbreviate the axioms of extensionality, empty set, pairing and union with $\mathsf{Ext}$, $\mathsf{Emp}$, $\mathsf{Pair}$, and $\mathsf{Un}$, respectively.

\begin{theorem}
    \label{Theorem: Extensible axioms and schemes}
    The following axioms and axiom schemes are ${\mathsf{Ext}}$-extensible: extensionality, empty set, pairing, union, $\in$-induction, separation and $\Delta_0$-separation, power set, replacement, and exponentiation. Moreover, the axiom of strong infinity is $\set{\mathsf{Ext},\mathsf{Emp},\mathsf{Pair},\mathsf{Un}}$-extensible.
\end{theorem}
\begin{proof}
    We will prove all statements of the theorem by assuming that a model $M$ satisfies the relevant axiom or scheme and then show that the extended model $M^+$ satisfies them as well. 
    
    For extensionality, let $x, y \in D_r$. For the non-trivial direction, assume that $r \Vdash \forall z (z \in x \leftrightarrow z \in y)$. By persistence and extensionality in $M$, we have $v \Vdash x = y$ for every $v \in M$; hence, $x(v) = y(v)$ for all $v \in M$. To see that also $x(r) = y(r)$, observe that $z \in x(r)$ if and only if $z E_r x(r)$. By assumption, the latter is equivalent to $z E_r y(r)$, which holds if and only if $z \in y(r)$. In conclusion, $y(r) = x(r)$.
    
    For the empty set axiom, let $e_v$ be the unique (by extensionality) witness for the empty set axiom at $v \in K$. Define a function $e$ with domain $K^+$ such that $e(v) = e_v$ for $e \in K$ and $e(r) = \emptyset$; by uniqueness $e$ is defined and it follows that $e \in D_r$. To see that $e$ witnesses the empty set axiom, let $x \in D_r$. We have to show that $r \Vdash \neg x \in e$. For this, it suffices to show that for all $v \geq r$, $v \not \Vdash x \in e$ but this is trivially true as $e(r)$ is empty and $e_v$ is the empty set for all $v \in K$.
    
    For the pairing axiom, let $x, y \in D_r$. By pairing and extensionality in $M$, there is a unique $p_v \in D_v$ such that $v \Vdash \forall z (z \in p_v \leftrightarrow z = x(v) \vee z = y(v))$ for all $v \in K$. Define $p$ to be the function with domain $K^+$ such that $p(r) = \set{x,y}$ and $p(v) = p_v$ for all $v \in K$. Clearly, $p \in D_r$ is defined by uniqueness of the $p_v$. To see that $p$ indeed witnesses the pairing axiom, observe that, clearly, $r \Vdash x \in p$ and $r \Vdash y \in p$. Moreover, if $r \Vdash z \in p$, then it follows by definition of $p$ that $r \Vdash z = x \vee z = y$.

    For the union axiom, let $x \in D_r$. As before, by extensionality and the union axiom in $M$, we can find a unique witness $u_v$ such that $v \Vdash u_v = \bigcup x(v)$ for every $v \in K$. Then define a function $u$ with domain $K^+$ such that $u(v) = u_v$ for all $v \in K$ and $u(r) = \bigcup \Set{y(r)}{y \in x(r)}$. To verify that indeed $u \in D_r$, note that $y \in u(r)$ implies that there is some $z \in x(r)$ such that $y \in z(r)$. Now by $y, z \in D_r$, we know that $v \Vdash y(v) \in z(v) \wedge z(v) \in x(v)$, so, clearly, $v \Vdash y(v) \in u(v)$. It is now a straightforward computation to see that $u$ witnesses the union axiom at $r$.
    
    Regarding $\in$-induction, suppose for contradiction that $r \not \Vdash \forall x (\forall y \in x \ \phi(y) \rightarrow \phi(x)) \rightarrow \forall x \phi(x)$. Then there is some $v \geq r$ such that $v \Vdash \forall x (\forall y \in x \ \phi(y) \rightarrow \phi(x))$ but $v \not \Vdash  \forall x \phi(x)$. As $M$ satisfies $\in$-induction, it must be that $v = r$, and so by persistence $M^+ \Vdash \forall x (\forall y \in x \phi(y) \rightarrow \phi(x))$. By $\in$-induction in $M$, $M \Vdash \forall x \phi(x)$, so all failures of this instance of $\in$-induction must happen at the new root $r$. Thus, as $r \not \Vdash \forall x \phi(x)$, it follows that there is some $x_0 \in D_r$ such that $r \not \Vdash \phi(x_0)$. Using the antecedent of $\in$-induction, this means that there must be some $x_1 \in D_r$ such that $r \Vdash x_1 \in x_0$ and $r \not \Vdash \phi(x_1)$. Iterating this construction, we obtain a sequence $\set{x_n}_{n \in \omega}$ such that $x_{n+1} \in x_{n}(r)$. This straightforwardly gives rise to an infinitely decreasing $\in$-chain. A contradiction.
    
    Next, we consider the separation axiom. Let $x \in D_r$ and $\phi(y)$ be a formula, possibly with parameters. Now, for every $v \in K$, let $s_v$ be the unique result of separating from $x(v)$ with $\phi$ and parameters $\bar p(v)$ at node $v$. With persistence and extensionality, it follows that the function $s$ with domain $K^+$ such that $s(v) = s_v$ and $s(r) = \Set{y \in x(r)}{r \Vdash \phi(y,\bar p)}$ is a well-defined element of $D_r$; if $\alpha$ is the least such that $x \in D_r^\alpha$, then it is easy to see that $x \in D_r^{\alpha + 1}$. It follows straightforwardly from the definition of $s$ that it witnesses separation from $x$ with $\phi$ at $r$: $r \Vdash z \in s$ is equivalent to $z \in s(r)$, and the latter holds by definition if and only if $r \Vdash \phi(z,\bar p)$. Note that the proof of $\Delta_0$-separation schema is a special case of the proof for the separation schema.
    
    For the power set axiom, consider $x \in D_r$ and let $\beta < \kappa$ such that $x \in D_r^\beta$. If $y \in D_r$ such that $r \Vdash y \subseteq x$, then $y(r) \subseteq x(r)$ and hence $y \in D_r^\beta$ as well. Let $p(r)$ consist of those $y \in D_r$ such that $r \Vdash \forall z (z \in y \rightarrow z \in x)$, then $p(r) \subseteq D_r^\beta$. Moreover, let $p(v)$ be the unique element of $D_v$ such that $v \Vdash p(v) = \mathcal{P}(x(v))$ (using extensionality). By persistence, $p$ is a well-defined element of  $D_r$. It is then straightforward to check that $p$ is the power set of $x$ at $r$. 
    
    For replacement, let $x \in D_r$ and $\phi$ be a formula (potentially with parameters) such that $r \Vdash \forall y \in x \exists ! z \phi(y,z)$. Given this, let $a(r)$ consist of those $z \in D_r$ for which there exists some $y \in D_r$ with $r \Vdash \phi(y,z)$. Moreover, let $a(v)$ be the witness for applying replacement with $\phi$ on $x(v)$ at $v$. By inaccessibility of $\kappa$ and persistence in $M^+$, it follows that $a \in D_r$. As in the previous cases, it is now straightforward to check that $a$ witnesses replacement.
    
    For exponentiation, let $a, b \in D_r$. Let $z(r)$ be the set of functions from $a$ to $b$ at $r$, and let $z(v)$ be the set of functions from $a(v)$ to $b(v)$ at the node $v$. It follows that $z \in D_r$, and it is easy to check that $z$ witnesses the exponentiation axiom.
    
    Recall that the axiom of strong infinity asserts the existence of a least inductive set. So, for every $v \in K$, let $\omega_v \in D_v$ be this least inductive set. At the new root $r$, we recursively construct sets $n_r$ as follows: Let $0_r$ be the empty set as defined from the empty set axiom. Then, given $n_r$, use pairing and union, to obtain $(n+1)_r$ such that $r \Vdash (n+1)_r = n_r \cup \set{n_r}$. As each $\omega_v$ is the least inductive set at $v$, it must be the case that $v \Vdash n_r(v) \in \omega_v$ for all $v \in K$. Therefore, the set $x$, defined by $x(r) = \Set{n_r}{n \in \omega}$ and $x(v) = \omega_v$ for $v \in K$, is a well-defined set at $r$, i.e. $x \in D_r$. By construction, we must have that if $r \Vdash \text{``$y$ is inductive''}$, then $r \Vdash x \subseteq y$. Hence $x$ witnesses strong infinity.
\end{proof}

A combination of \Cref{Theorem: Extensible axioms and schemes} and \Cref{Lemma: Extensible 1,Lemma: Extensible 2} yields the next corollary. An application of \Cref{Theorem: Extensibility implies admissible rules} then yields \Cref{Corollary: admissible rules of IZFR CZFER ECST BCST}.

\begin{corollary}
    \label{Theorem: IZFR CZFER ECST BCST extensible}
    The theories $\IZFR$, $\CZFER$, $\ECST$ and $\BCST$ are extensible. \qed
\end{corollary}

\begin{corollary}
    \label{Corollary: admissible rules of IZFR CZFER ECST BCST}
    The propositional admissible rules of $\IZFR$, $\CZFER$, $\ECST$, and $\BCST$ are exactly those of $\IPC$. In other words, if $T$ is one of these theories, then ${\vsim_T} = {\vsim_\IPC}$. \qed
\end{corollary}

\begin{remark}
    It seems difficult to use the method above to prove that the axiom schemes of collection, strong collection and subset collection are extensible. The reason for this is that these axiom schemes do not require that their witnesses are \emph{unique}---this is in contrast to their weakened versions, \emph{replacement} and \emph{exponentiation}. 
\end{remark}

\begin{question}
    Are the axiom schemes of collection, strong collection and subset collection extensible?
\end{question}

\begin{remark}
    \label{Remark: Alternatives for inaccessibles}
    It is not strictly necessary to assume the existence of a proper class of inaccessible cardinals as (at least) the following two alternatives are available: First, we could work with Kripke models that have (definable) class domains. Second, we could apply the downwards Löwenheim-Skolem Theorem (which does hold in our classical metatheory) to work, without loss of generality, only with models with countable domains. In our view, the solution with inaccessible cardinals is the most elegant as it allows us to ignore any worries about size restrictions.
\end{remark}

\section{Questions}
\label{Section: questions}

We close with a few suggestions for future research.

\begin{question}
    \label{Question: T T' vsimpr}
    Is it the case that ${\vsimpr_{\HA}} = {\vsimpr_{\IZF}}$? 
\end{question}

In general, one can ask the same question for any two constructive theories $T$, $T'$ of interest.

\begin{question}
    Are $\CZF$ and $\IZF$ extensible?
\end{question}

\begin{question}
    Of course, both $\IZF$ and $\CZF$ can be extended with further axioms. Natural examples of such further axioms include the so-called \textit{large set axioms}---the constructive counterparts to large cardinal axioms. What are the propositional and first-order logics of extensions of $\IZF$ and $\CZF$ with further axioms? Also, what are their admissible rules? 
\end{question}

\begin{question}
    In this paper we gave a first \textit{structural} analysis of the logical structure by observing the consequences of the extension property. Are there other structural properties of set theories that determine (parts of) the logical structure of a given set theory?
\end{question}

\begin{question}
    Let $T$ be a theory. Given a class $\Gamma$ of set-theoretic formulas, we can obtain the notions of $\Gamma$-propositional logic $\Logic(T)$ and  $\Gamma$-first-order logic $\QLogic(T)$, as well as $\Gamma$-admissible rules by restricting to those $T$-substitutions that arise from $T$-assignments with domain in $\Gamma$. Natural classes to consider are, of course, the classes of the Lévy hierarchy but other cases are thinkable (such as prenex formulas). What are the logics that arise in this way?
\end{question}

Note, regarding the previous question, that the substitutions used to prove the results in \Cref{Section: techniques and results} are usually of rather low complexity, e.g. $\Sigma_3$ or $\Sigma_4$. Therefore, the question is often only interesting for classes of lower complexity. For instance, Visser \cite[Section 3.7-3.9]{Visser1999} considers $\Sigma_1$-substitutions.

In this paper, we have only considered set theories on the basis of intuitionistic logic. Of course, one can also analyse the logical structure of set theories based on other logics. Löwe, Passmann and Tarafder \cite{LowePassmannTarafder2021} take some first steps towards an analysis of the logical structure of certain paraconsistent set theories. Another related problem is to determine the provability logics of constructive and intuitionistic set theories.

\section{The classical case}
\label{Section: classical case}

In this final section we briefly discuss the logical structure of classical theories. Recall that a logic is \textit{Post-complete} if it has no consistent extension.

\begin{theorem}
    \label{Theorem: classical theory propositional de jongh}
    Let $T$ be a classical theory. If $T$ is consistent, then $T$ satisfies the de Jongh property for $\CPC$.
\end{theorem}
\begin{proof}
    Since $T$ is classical, it must be that $\CPC \subseteq \Logic(T)$. As $T$ is consistent, $\bot \notin \Logic(T)$. Hence, $\Logic(T) = \CPC$ as $\CPC$ is Post-complete.  
\end{proof}

The case for first-order logics of a given theory is more complicated as they are not Post-complete; Yavorsky \cite{yavorsky1997logical} proved that expressively strong arithmetical theories $T$, such as $\mathsf{PA}$, satisfy $\QLogic(T) = \CQC$ (see Yavorsky's paper for definitions and details). Yavorsky also shows that there are classical theories $T$ such that $\CQC \subsetneq \QLogic(T)$.

\begin{theorem}
    \label{Theorem: Propositional admissible rules preserved}
    Let $T$ be a classical theory. Then $A \vsim_T B$ if and only if $A \vsim_{\Logic(T)} B$. Thus  ${\vsim_T}\ =\ \vsim_{\Logic(T)}\ =\ {\vdash_{\CPC}}$. 
\end{theorem}
\begin{proof}
Theorem~\ref{Theorem: classical theory propositional de jongh} implies that $\Logic(T)=\CPC$. As is well-known, $\vsim_{\CPC}\ =\ \vdash_{\CPC}$ (a proof can be found in \cite{Iemhoff2016}). Thus the second part of the theorem follows from the first. 

For the first part, 
the direction from left to right has been proven in Theorem~\ref{Theorem: Bound on admissible rules}. 
We treat the other direction. 
Arguing by contraposition, assume for some propositional formulas $A$ and $B$ that $\vdash_T\rho(A)$ and $\nvdash_T \rho(B)$ for some substitution $\rho$ from the language of propositional logic to the language of $T$. We have to show that $A\not\vsim_{\Logic(T)}B$. Assumption $\nvdash\rho(B)$ implies that there exists a propositional model $M$ of $T$ and a valuation $v$ such that $M,v\not\models\rho(B)$. Define a substitution $\sigma$ in propositional logic as follows:
\[\sigma(p)=
\begin{cases}
\top &\text{if }M,v\models\rho(p),\\
\bot &\text{if }M,v\not\models\rho(p).
\end{cases}
\]
It is easy to see that for any propositional formula $C$, the formula $\sigma C$ contains no propositional variables, which implies that $\sigma C$ is equivalent to $\top$ or to $\bot$. 
Therefore, 
\begin{equation}\label{excl_mid_eq}
\vdash_{\Logic(T)}\sigma (C)\text{ or }\vdash_{\Logic(T)}\neg\sigma (C).
 \end{equation}
Using \eqref{excl_mid_eq} it is easy to prove by induction on the propositional formula $C$ that 
\begin{equation}\label{P2Qadm_eq}
\vdash_{\Logic(T)}\sigma (C)\iff M,v\models\rho(C).
\end{equation}
Now, since $M\not\models \rho(B)$, it follows that $\nvdash_{\Logic(T)}\sigma (B)$ from (\ref{P2Qadm_eq}). On the other hand, since $\vdash_T\rho (A)$, we must have that $M\models\rho (A)$, thus (\ref{P2Qadm_eq}) implies that $\vdash_{\Logic(T)}\sigma (A)$. Therefore
$A\not\vsim_{\Logic(T)} B$, which is what we had to show.
\end{proof}

In the remainder of this  section we show that the first part of Theorem~\ref{Theorem: Propositional admissible rules preserved} can be lifted to the predicate case, at least for theories for which $\QLogic(T)=\CQC$, and the second part cannot. 
The predicate version of the first part of the previous theorem states that $\vsimpr_T \ =\ \vsimpr_{\CQC}$ and the second part states that $\vsimpr_{\CQC} \ =\ \vdash_{\CQC}$. To show that the first statement is true for theories $T$ such that $\QLogic(T)=\CQC$, we first have to prove that the second statement, while not true, is almost true; we have to show that $\CQC$ is {\em almost structurally complete}. 

As remarked above, $\CPC$ is {\em structurally complete}, which means that any admissible rule is derivable: for all propositional formulas $A$ and $B$
\begin{equation}
A\vsim_{\CPC}B \Rightarrow A \vdash_{\CPC} B. 
\end{equation}
In other words, $\vsim_{\CPC}\ =\ \vdash_{\CPC}$ (see also Iemhoff \cite{Iemhoff2016}).

It is easy to see that, unlike $\CPC$, $\CQC$ is not structurally complete. Consider the formula $\E xP(x) \wedge \E x\neg P(x)$ and observe that for no substitution $\sigma$ the formula $\sigma (\E xP(x) \wedge \E x \neg P(x))$ is derivable in $\CQC$. For if it were, it would hold in models consisting of one element, quod non. This implies that $\E xP(x) \wedge \E x\neg P(x) \vsimpr_{\CQC} \bot$. And since $\E xP(x) \wedge \E x\neg P(x) \not \vdash_{\CQC} \bot$, this shows that $\CQC$ is not structurally complete, i.e.\ that $\vsimpr_{\CQC} \ \neq\ \vdash_{\CQC}$. However, the admissible rule $\E xP(x) \wedge \E x\neg P(x) \vsimpr_{\CQC} \bot$ has a particular property, it is {\em passive}. 

\begin{definition}
    Admissible rules $A \vsim_T B$ are called {\em passive} if the premise $A$ is not \textit{unifiable}, i.e.\ $\vdash_T \sigma (A)$ for no substitution $\sigma$. 
\end{definition}

We will show that the only non-derivable admissible rules in $\CQC$ are passive. In the terminology of admissibility: $\CQC$ is {\em almost structural complete}. This result was first obtained by Pogorzelski and Prucnal \cite{PogorzelskiPrucnal1975}. Here we prove a modest generalisation that applies to a class of consequence relations instead of a single one.
Recall that a rule $A/B$ is derivable in $\CQC$ if $A \vdash B$, where $\vdash$ is a consequence relation of $\CQC$, by which we mean a consequence relation that has the same theorems as $\CQC$: $\vdash A$ iff $A$ holds in $\CQC$. Thus, strictly speaking, the notion of derivable rule depends on the given consequence relation of $\CQC$, which is why we make it explicit in the discussion below. Since the derivability of a formula does not depend on the consequence relation (because all have the same derivable formulas), we can leave it implicit in such a setting, as in the next lemma, where $\vdash_{\CQC}$ denotes any consequence relation for $\CQC$.

\begin{lemma}
 \label{Lemma: Ground substitutions}
A formula $A$ is unifiable if and only if there exists a ground substitution $\tau$ such that $\vdash_{\CQC}\tau(A)$ and $\tau (P(\bar{s})) =\tau (P(\bar{t}))$ for all predicates $P$ and sequences of terms $\bar{s},\bar{t}$.
\end{lemma}
\begin{proof}
It suffices to show the direction from left to right. Since $A(\bar{x})$ is unifiable, there exists a substitution $\sigma$ such that $\vdash_{\CQC}\sigma(A)$, from which it follows that $\vdash_{\CQC}\rho\sigma(A)$ for any ground substitution $\rho$, as $\CQC$ is closed under uniform substitution. 
Clearly, $\rho\sigma$ is a ground substitution too. That $\rho\sigma$ satisfies the second part of the lemma follows immediately from the definition of substitutions. 
\end{proof}

\begin{theorem}
    \label{Theorem: CQC almost SC}
    Let $\vdash_{\CQC}$ be a consequence relation for $\CQC$ such that  $C(\bar{x})\vdash_{\CQC}\forall \bar{x}C(\bar{x})$ for all formulas $C$. Then $A\vsimpr_{\CQC} B$ implies $A\vdash_{\CQC} B$ for all unifiable formulas $A$ and all formulas $B$.
\end{theorem}
\begin{proof}
Since $A$ is unifiable there exists, by Lemma~\ref{Lemma: Ground substitutions}, a ground substitution $\tau$ such that $\vdash\tau(A)$ and $\tau (P(\bar{s})) =\tau (P(\bar{t}))$ for all predicates $P$ and sequences of terms $\bar{s},\bar{t}$. First note that this implies that $\tau (B(\bar{s})) =\tau (B(\bar{t}))$ for any formula $B$. Let $\bar{x}$ be the free variables in $A$ and define a substitution $\sigma$ as follows
\[
\sigma(P\bar{z})=
\begin{cases}
\forall \bar{x}A\to P(\bar{z})    & \text{ if }\vdash\tau P(\bar{z})\leftrightarrow \top,\\
\forall \bar{x}A\wedge P(\bar{z}) & \text{ if }\vdash\tau P(\bar{z})\leftrightarrow \bot.
\end{cases}
\]

\vspace{6pt}

\begin{claim*}
    For any formula $B$, we have that
    \[
    \vdash\sigma(B)\leftrightarrow
    \begin{cases}
    \forall \bar{x}A\to B    & \text{ if }\vdash\tau  B\leftrightarrow \top,\\
    \forall \bar{x}A\wedge B & \text{ if }\vdash\tau  B\leftrightarrow \bot.
    \end{cases}
    \]
\end{claim*}
\begin{proof}[Proof of the Claim]
\renewcommand{\qedsymbol}{$\dashv$}
By induction on $B$. 

\begin{description}
    \item[Case $B=P\bar{z}$.] Trivial.

    \item[Case $B= C\wedge D$.] If $\vdash\tau( C\wedge D)\leftrightarrow\top$, then as $\tau$ commutes with connectives, we must have that $\vdash\tau C\leftrightarrow\top$ and $\vdash\tau D\leftrightarrow\top$, which by the induction hypothesis imply that 
    \[\vdash\sigma C\leftrightarrow(\forall \bar{x}A\to C)\quad\text{ and }\quad\vdash\sigma D\leftrightarrow(\forall \bar{x}A\to D).\]
    Thus $\sigma( C\wedge D)\dashv\vdash(\forall \bar{x}A\to C)\wedge (\forall \bar{x}A\to D)\dashv\vdash\forall \bar{x}A\to( C\wedge D)$.

    If $\vdash\tau( C\wedge D)\leftrightarrow\bot$, then there are three cases.

\begin{description}
    \item[Case 1.] $\vdash\tau C\leftrightarrow\bot$ and $\vdash\tau D\leftrightarrow\bot$. By induction hypothesis, $\vdash\sigma C\leftrightarrow(\forall \bar{x}A\wedge C)$ and $\vdash\sigma D\leftrightarrow(\forall \bar{x}A\wedge D)$. Thus, $\vdash\sigma( C\wedge D)\leftrightarrow (\forall \bar{x}A\wedge( C\wedge D))$.

    \item[Case 2.] $\vdash\tau C\leftrightarrow\bot$ and $\vdash\tau D\leftrightarrow\top$. By induction hypothesis, $\vdash\sigma C\leftrightarrow(\forall \bar{x}A\wedge C)$ and $\vdash\sigma D\leftrightarrow(\forall \bar{x}A\to D)$. Thus, we have that
\[\sigma C\wedge\sigma D\dashv\vdash(\forall \bar{x}A\wedge C)\wedge(\forall \bar{x}A\to D)\dashv\vdash\forall \bar{x}A\wedge( C\wedge D).\]

    \item[Case 3.] $\vdash\tau C\leftrightarrow\top$ and $\vdash\tau D\leftrightarrow\bot$. Analogous to Case 2.
\end{description}

    \item[Case $B=\neg C$.] If $\vdash\tau(\neg C)\leftrightarrow\top$, then since $\tau$ commutes with the connectives we must have that $\vdash\tau C\leftrightarrow\bot$, which by induction hypothesis implies that $\sigma C\dashv\vdash\forall \bar{x}A\wedge C$. Thus, \[\sigma(\neg  C)\dashv\vdash\neg(\forall \bar{x}A\wedge C)\dashv\vdash \forall\bar{x}A\to\neg C.\]

    If $\vdash\tau(\neg C)\leftrightarrow\bot$, then $\vdash\tau C\leftrightarrow\top$ since $\tau$ commutes with the connectives, which by induction hypothesis implies that $\sigma C\dashv\vdash\forall \bar{x}A\to C$. Thus, \[\sigma(\neg  C)\dashv\vdash\neg(\forall \bar{x}A\to C)\dashv\vdash \forall\bar{x}A\wedge\neg C.\]
    
    \item[Case $B=\forall \bar{w} C(\bar{w})$.] If $\vdash\tau(\forall \bar{w} C(\bar{w}))\leftrightarrow\top$, then $\vdash \forall \bar{w}\tau ( C(\bar{w}))\leftrightarrow\top$ follows. 
    Hence $\vdash\tau C(\bar{w})\leftrightarrow\top$, which by induction hypothesis implies that $\sigma C(\bar{w})\dashv\vdash\forall \bar{x}A\to C(\bar{w})$. Thus, \[\sigma(\forall \bar{w} C(\bar{w}))\dashv\vdash \forall \bar{w}\sigma C(\bar{w})\dashv\vdash\forall\bar{w}(\forall \bar{x}A(\bar{x})\to C(\bar{w}))\dashv\vdash \forall\bar{x}A(\bar{x})\to\forall\bar{w} C(\bar{w}).\]

    If $\vdash\tau(\forall \bar{w} C(\bar{w}))\leftrightarrow\bot$, then the fact that $\tau$ is a substitution implies that $\vdash \tau(\forall \bar{w} C(\bar{w})) \leftrightarrow \forall \bar{w}\tau(C(\bar{w}))$, and that it is a ground substitution that $\vdash \tau C(\bar{w})\leftrightarrow\bot$ for some $\bar{w}$. By the observation made in the first paragraph of the proof, $\tau  C(\bar w)$ is equivalent to $\tau C(\bar s)$ for any sequence of terms $\bar s$. 
Thus 
    $\sigma C(\bar{w}) \dashv\vdash\forall \bar{x}A\wedge C(\bar{w})$ by the induction hypothesis, which implies 
    \[\sigma(\forall \bar{w} C(\bar{w}))\dashv\vdash\forall\bar{w}(\forall \bar{x}A(\bar{x})\wedge C(\bar{w}))\dashv\vdash \forall\bar{x}A(\bar{x})\wedge\forall\bar{w} C(\bar{w}).\]
\end{description}
This proves the claim.
\end{proof}

Now, since $\vdash\tau A\leftrightarrow\top$, we know by the claim that $\vdash\sigma A(\bar{x})\leftrightarrow(\forall\bar{x}A(\bar{x})\to A)$. But $\vdash\forall\bar{x}A(\bar{x})\to A$, thus $\vdash\sigma A$. Since $A\vsimpr_{\CQC} B$, we have that $\vdash\sigma B$.

By the claim, it is easy to see that $\forall \bar{x}A(\bar{x})\vdash\sigma B\leftrightarrow B$. By assumption $A\vdash_{\CQC}\forall \bar{x}A(\bar{x})$, we obtain that $A\vdash B$.
\end{proof}

\begin{corollary}
$\CQC$ is almost structurally complete with respect to any consequence relation that satisfies the constraint in Theorem~\ref{Theorem: CQC almost SC}.
\end{corollary}

\begin{corollary}
 \label{Corollary: Admissibility classical theories}
    For any classical theory $T$ such that $\QLogic(T)=\CQC$: $A\vsimpr_TB$ if and only if $A\vsimpr_{\CQC}B$, i.e.\ $\vsimpr_T\ =\ \vsimpr_{\CQC}$.
\end{corollary}
\begin{proof}
 The assumption $\QLogic(T)=\CQC$ gives $A\vsimpr_T B$ implies $A \vsimpr_{\CQC}B$ by Theorem~\ref{Theorem: Bound on admissible rules}. For the other direction, assume that $A \vsimpr_{\CQC}B$. If $A$ is unifiable in $\CQC$, then $A\vdash_{\CQC}B$ by Theorem~\ref{Theorem: CQC almost SC}. Hence $A \vdash_TB$ and thus $A \vsimpr_TB$. If $A$ is not unifiable in $\CQC$, then it suffices to show that $A$ is not unifiable in $T$. For this implies that $A\vsimpr_TC$ for any $C$, and thus $A\vsimpr_TB$ would follow. Arguing by contradiction, suppose $A$ is unifiable in $T$, say $\vdash_T\sigma A$ for some substitution $\sigma$. Then for any ground substitution $\tau$, $\vdash_T\tau\sigma A$, and as $\QLogic(T)=\CQC$, this implies $\vdash_{\CQC}\tau\sigma A$, contradicting the assumption that $A$ is not unifiable in $\CQC$. 
\end{proof}

\funding{The first author gratefully acknowledges support by the Netherlands Organisation for Scientific Research under grant 639.073.807 and by the MOSAIC project (EU H2020-MSCA-RISE-2020 Project 101007627).
The research of the second author was partially funded by a doctoral scholarship of the \textit{Studienstiftung des deutschen Volkes} (German Academic Scholarship Foundation).}

\ack{We would like to thank two anonymous reviewers for their very helpful comments.}

\bibliographystyle{rsta}
\bibliography{bibliography}

\end{document}